\documentclass[reqno]{amsart}
\usepackage{amssymb}
\usepackage{amsmath}
\usepackage{hyperref}

\newtheorem{theorem}{Theorem}[section]
\newtheorem{lemma}[theorem]{Lemma}
\newtheorem{proposition}[theorem]{Proposition}
\newtheorem{corollary}[theorem]{Corollary}
\theoremstyle{definition}
\newtheorem{definition}[theorem]{Definition}
\newtheorem{example}[theorem]{Example}

\theoremstyle{remark}
\newtheorem{remark}[theorem]{Remark}

\numberwithin{equation}{section}
\begin{document}

\title{On Linear Dynamics of Sets of Operators}
\author{ Mohamed Amouch and Otmane Benchiheb }

\address{\textsc{Mohamed Amouch and Otmane Benchiheb},
University Chouaib Doukkali.
Department of Mathematics, Faculty of science
Eljadida, Morocco}
\email{amouch.m@ucd.ac.ma}
\email{otmane.benchiheb@gmail.com}

\begin{abstract}
Let $X$ be a complex topological vector space with $dim(X)>1$ and $\mathcal{B}(X)$ the set of all continuous linear operators on $X$.
The concept of hypercyclicity for a subset of $\mathcal{B}(X)$, was introduced in \cite{AKH}.
In this work, we introduce the notion of hypercyclic criterion for a subset of $\mathcal{B}(X)$.
We extend some results known for a single operator and $C_0$-semigroup to a subset of $\mathcal{B}(X)$
and we give applications for $C$-regularized groups of operators.
\end{abstract}


\maketitle

\section{Introduction and Preliminary}
Let $X$ be a complex topological vector space with $dim(X)>1$ and $\mathcal{B}(X)$ the set of all continuous linear operators on $X$. By an operator, we always mean a continuous linear operator.
If $T\in\mathcal{B}(X),$ then the orbit of a vector $x\in X$ under $T$ is the set
$$ Orb(T,x):=\{T^{n}x\mbox{ : }n\in\mathbb{N}\}. $$
An operator $T\in\mathcal{B}(X)$ is said to be hypercyclic if there is some vector $x\in X$ such that $Orb(T,x)$ is dense in $X;$
such a vector $x$ is called a hypercyclic vector for $T.$ The set of all hypercyclic vectors for $T$ is denoted by $HC(T).$
Recall \cite{Birkhoff} that $T\in\mathcal{B}(X)$ is said to be topologically transitive if for each pair
$(U,V)$ of nonempty open subsets of $X,$ there exists $n\in\mathbb{N}$ such that
$$T^{n}(U)\cap V\neq \emptyset.$$

The hypercyclicity criterion for a single operator was introduced in \cite{Bes, Gethner Shapiro, Kitai}.
It provides several sufficient conditions that ensure hypercyclicity.
We say that an operator $T\in\mathcal{B}(X)$ satisfies the hypercyclicity criterion if there exist an increasing sequence of integers $(n_k)$, two dense sets $X_0$, $Y_0\subset X$ and a sequence of maps $S_{n_k}$ : $Y_0\longrightarrow X$ such that$:$
\begin{itemize}
\item[$(i)$] $T^{n_k}x\longrightarrow 0$ for any $x\in X_0$;
\item[$(ii)$] $S_{n_k}y\longrightarrow 0$ for any $y\in Y_0$;
\item[$(iii)$] $T^{n_k} S_{n_k}y\longrightarrow y$ for any $y\in Y_0$.
\end{itemize}
For a general overview of hypercyclicity and related properties in linear dynamics,
see \cite{Bayart Matheron, BMP,Feldman, Godefroy Shapiro,GED,Erdmann Peris, GEU, MS,Muller}.

Recall \cite[Definition 1.5.]{Erdmann Peris}, that an operator $T\in\mathcal{B}(X)$ is called quasi-conjugate or quasi-similar to an operator $S\in\mathcal{B}(Y)$ if there exists a continuous map $\phi$ : $Y\longrightarrow X$ with dense range such that $T\circ\phi=\phi\circ S.$ If $\phi$ can be chosen to be a homeomorphism, then $T$ and $S$ are called conjugate or similar.
Recall \cite[Definition 1.7.]{Erdmann Peris}, that a property $\mathcal{P}$ is said to be preserved under quasi-similarity if the following holds$:$ if an operator $T\in\mathcal{B}(X)$ has property $\mathcal{P}$, then every operator $S\in\mathcal{B}(Y)$ that is quasi-similar to $T$ has also property $\mathcal{P}$.
If $\Gamma\subset \mathcal{B}(X)$ and $x\in X$, then the orbit of $x$ under $\Gamma$ is the set
$$Orb(\Gamma,x)=\{Tx\mbox{ : }T\in\Gamma\}.$$
The notion of hypercyclic set of operators, was introduced in  \cite{AKH}:
A subset $\Gamma$ of $\mathcal{B}(X)$ is said to be hypercyclic set of operators or hypercyclic set,
if there exists some $x\in X$ such that $Orb(\Gamma,x)$ is a dense subset of $X;$
such vector $x$ is called a hypercyclic vector for $\Gamma$ or hypercyclic vector.
The set of all hypercyclic vectors for $\Gamma$ is denoted by $HC(\Gamma)$.\\
It is clear that $T\in\mathcal{B}(X)$ is a hypercyclic operator if and only if $\Gamma=\{T^n\mbox{ : }n\in\mathbb{N}\}$
is a hypercyclic set. In this case, we write $Orb(T,x)$ instead of $Orb(\Gamma,x).$

Recall that a subset of a topological space $X$ is called a $G_\delta$ set
if and only if it is a countable intersection of open sets, see\cite{Willered}.

In \cite{Birkhoff}, Birkhoff has shown that the set of hypercyclic vectors of a single operator is a $G_\delta$ set.
In section 2, we show that this result holds for the set of hypercyclic vectors of a set of operators.
 We also introduce the notion of quasi-similarity for sets of operators. We prove that hypercyclicity 
 for a set of operators is preserved under quasi-similarity.
In section 3, we study the notion of topologically transitive set introduced in \cite{AKH}. In addition,
we introduce the notion of hypercyclic criterion for sets of operators and we
give relations between these two notions and the concept of hypercyclicity for sets of operators.
We also prove that the topological transitivity is preserved under quasi-similarity.
In section 4, we give applications for $C$-regularized group of operators. We prove that, if $(S(z))_{z\in\mathbb{C}}$ is a hypercyclic $C$-regularized group of operators and $C$ has a dense range, then $(S(z))_{z\in\mathbb{C}}$ is topologically transitive.
\section{Hypercyclic Sets of Operators}
We begin this section by given an example of hypercyclic set of operators. This example shows that there exists a hypercyclic set that is not of the form $\{T^n\mbox{: }n\in\mathbb{N}\}$, where $T\in\mathcal{B}(X)$.
\begin{example}\label{exa1}
Let $X=\ell^{2}(\mathbb{N})$. For all polynomial $p=a_0+a_1x+\dots +a_n x^n$ with $a_0\dots a_n\in\mathbb{C}$
and $n\in\mathbb{N}$. Let $T_p$ be an operator defined on $\ell^2(\mathbb{N})$ by
$$ \begin{array}{ccccc}
T_p & : & \ell^2(\mathbb{N}) & \longrightarrow & \ell^2(\mathbb{N}) \\
 & & (x_k)_{k\in\mathbb{N}} & \longmapsto & (a_0x_0,a_1x_1,\dots,a_n x_{n},0,\dots). \\
\end{array} $$
Let $e=(e_k)_{k\in\mathbb{N}}\in\ell^2(\mathbb{N})$ an orthonormal basis of $\ell^2(\mathbb{N})$ and $\Gamma=\{T_p \mbox{ : }p \mbox{ polynomial }\}$, then
$$  Orb(\Gamma,x)=\{(a_0e_0,a_1e_1,\dots,a_n e_{n},0,\dots)\mbox{ : } a_0,\dots,a_n\in\mathbb{C} \mbox{, } n\in\mathbb{N}\}=\mbox{span}\{e_n \mbox{ : }n\in\mathbb{N}\} . $$
Hence, $\overline{Orb(\Gamma,x)}=\overline{\mbox{span}\{ Orb(\Gamma,x) \}}=\ell^2(\mathbb{N})$, and so $\Gamma$ is a hypercyclic set of operators.
\end{example}
Let $T\in\mathcal{B}(X)$ be a hypercyclic operator. Bourdon and Feldman \cite{Bourdon Feldman} proved that if $x\in HC(T)$, then $Tx\in HC(T)$ and if $p$ is a nonzero polynomial, then $p(T)$ has dense range. These results does not hold for every hypercyclic set of operators. Indeed, if $\Gamma$ is defined as in Example \ref{exa1}, then $(e_k)_{k\in\mathbb{N}}$ is a hypercyclic vector for $\Gamma$, but $T_p((e_k)_{k\in\mathbb{N}})$ is not a hypercyclic for $\Gamma$, for all polynomial $p$. Moreover, $T_p$ is not of dense range, for all polynomial $p$. However, we have the following result.

Let $\Gamma\subset \mathcal{B}(X)$. We denote by $\{\Gamma\}^{'}$ the set of all elements of $\mathcal{B}(X)$ which commute with every element of $\Gamma$.
\begin{proposition}\label{prop1}
Let $\Gamma\subset\mathcal{B}(X)$ be a hypercyclic set and $T\in\mathcal{B}(X)$ be an operator with dense range. 
If $T\in\{\Gamma\}^{'}$, then $Tx\in HC(\Gamma)$ for all $x\in HC(\Gamma).$
\end{proposition}
\begin{proof}
Let $O$ be a nonempty open subset of $X$. Since $T$ is continuous and and of dense range, $T^{-1}(O)$ is a nonempty open subset of $X$. Let $x\in HC(\Gamma)$, then there exists $S\in\Gamma$ such that $ Sx\in T^{-1}(O)$, that is $ T(Sx)\in O$. Since $T\in\{\Gamma\}^{'}$, it follows that
$$ S(Tx)= T(Sx)\in O.$$
Hence $Orb(\Gamma,Tx)$ meets every nonempty open subset of $X$;
consequently, $Orb(\Gamma,Tx)$ is dense in $X$, which implies that $Tx\in HC(\Gamma).$
\end{proof}
\begin{corollary}
Let $\Gamma\subset\mathcal{B}(X)$ be a hypercyclic set. If $x\in HC(\Gamma)$, then $\alpha x\in HC(\Gamma)$, for all $\alpha\in\mathbb{C}\setminus\{0\}.$
\end{corollary}
\begin{proof}
Let $\alpha \in \mathbb{C}\setminus\{0\}$ and $x\in HC(\Gamma)$. Then $T=\alpha I$ is a continuous map with dense range and $T\in\{\Gamma\}^{'}$. Hence, by Proposition \ref{prop1}, $\alpha x\in HC(\Gamma)$.
\end{proof}
In the following Definition, we introduce the notion of quasi-similarity for sets of operators.
\begin{definition}
Let $X$ and $Y$ be topological vector spaces and let $\Gamma\subset \mathcal{B}(X)$ and $\Gamma_1\subset \mathcal{B}(Y)$. We say that $\Gamma$ and $\Gamma_1$ are quasi-similar if there exists a continuous map $\phi$ : $X\longrightarrow Y$ with dense range such that for all $T\in\Gamma,$ there exists $S\in\Gamma_1$ satisfying $S\circ\phi=\phi\circ T$. If $\phi$ is
a homeomorphism, then $\Gamma$ and $\Gamma_1$ are called similar.
\end{definition}
In \cite{HL}, Herrero showed that the hypercyclicity of operators is preserved under quasi-similarity, see also \cite[Proposition 2.24]{Erdmann Peris}.
In the following proposition, we prove that this result holds for sets of operators.

\begin{proposition}\label{14}
Let $X$ and $Y$ be topological vector spaces and let $\Gamma\subset\mathcal{B}(X)$ be quasi-similar to $\Gamma_1\subset\mathcal{B}(Y)$.
If $\Gamma$ is hypercyclic in $X$, then $\Gamma_1$ is hypercyclic in $Y$. Furthermore, 
$$\phi(HC(\Gamma))\subset HC(\Gamma_1).$$
\end{proposition}
\begin{proof}
Let $O$ be a nonempty open subset of $Y$, then $\phi^{-1}(O)$ is a nonempty open subset of $X$. Let $x\in HC(\Gamma)$, then there exists $T\in\Gamma$ such that $Tx\in \phi^{-1}(O)$, that is $\phi(Tx)\in O$. Let $S\in\Gamma_1$ such that $S\circ\phi=\phi\circ T$, then
$S(\phi (x)) =\phi(Tx)\in O$.
Thus, $Orb(\Gamma_1,\phi (x))$ meets every nonempty open set of $Y$,
that is $Orb(\Gamma_1,\phi (x))$ is dense in $Y$. Hence, $\Gamma_1$ is hypercyclic and 
$\phi (x)\in HC(\Gamma_1).$
\end{proof}
\begin{corollary}
Let $X$ and $Y$ be topological vector spaces and let $\Gamma\subset\mathcal{B}(X)$ be similar to $\Gamma_1\subset\mathcal{B}(Y)$.
If $\Gamma$ is hypercyclic in $X$, then $\Gamma_1$ is hypercyclic in $Y$. Furthermore,  
$$\phi(HC(\Gamma))=HC(\Gamma_1).$$
\end{corollary}
 The direct sum of two hypercyclic operators is not in general a hypercyclic operator. Indeed, Salas \cite{Salas}, De la Rosa and Read \cite{DR} and Herrero \cite{HL} showed that there exist $T_1$ and $T_2$ hypercyclic operators such that the direct sum $T_1\oplus T_2$ is not hypercyclic. However, if $T_1\oplus T_2$ is hypercyclic, then $T_1$ and $T_2$ are hypercyclic, see \cite[Proposition 2.25]{Erdmann Peris}.
In the following proposition, we prove that this result holds for sets of operators.

Let $\{X_i\}_{i=1}^{n}$ be a family of complex topological vector spaces and let $\Gamma_i$ be a subset of $\mathcal{B}(X_i)$, for all $1\leq i\leq n$. Define
$$\oplus_{i=1}^nX_i=X_1\times X_2\times\dots \times X_n=\{(x_1,x_2,\dots,x_n) \mbox{ : }x_i\in X_i\mbox{, }1\leq i\leq n\}$$
and
$$\oplus_{i=1}^n\Gamma_i=\{T_1\times T_2\times\dots \times T_n\mbox{ : }T_i\in\Gamma_i\mbox{, }1\leq i\leq n\}.$$
\begin{proposition}
Let $\{X_i\}_{i=1}^{n}$ be a family of complex topological vector spaces and $\Gamma_i$ a subset of
$\mathcal{B}(X_i),$ for all $1\leq i\leq n$. If $\oplus_{i=1}^n\Gamma_i$ is a hypercyclic set in $\oplus_{i=1}^n X_i$, then $\Gamma_i$ is a hypercyclic set in $X_i$, for all $1\leq i\leq n$. Moreover, if $(x_1,x_2,\dots,x_n)\in HC(\oplus_{i=1}^n\Gamma_i)$, then $x_i\in HC(\Gamma_i)$, for all $1\leq i\leq n$.
\end{proposition}
\begin{proof}
Let $(x_1,x_2,\dots,x_n)\in HC(\oplus_{i=1}^n\Gamma_i)$. If $O_i$ be a nonempty open subset of $X_i$
for all $1\leq i\leq n$, then $ O_1\times O_2\times\dots\times O_n$ is a nonempty open subset of $\oplus_{i=1}^n X_i$. Since $Orb(\oplus_{i=1}^n\Gamma_i,\oplus_{i=1}^n x_i)$ is dense in $\oplus_{i=1}^n X_i$, there exist $T_i\in\Gamma_i$ such that
$$(T_1\times T_2\times\dots \times T_n)(x_1,x_2,\dots,x_n)=(T_1 x_1,T_2 x_2,\dots,T_n x_n) \in  O_1\times O_2\times\dots\times O_n,$$
that is $T_i x_i\in O_i$, for all $1\leq i\leq n$. Hence, $\Gamma_i$ is a hypercyclic set in $X_i$ and $x_i\in HC(\Gamma_i)$, for all $1\leq i\leq n$.
\end{proof}
 In \cite{Birkhoff}, Birkhoff has shown that the set of hypercyclic vectors of a single operator is a $G_\delta$ set. In what follows, we prove that the same result holds for the set of hypercyclic vectors of a set of operators.
\begin{proposition}\label{5}
Let  $X$ be a second countable Baire  complex topological vector space and $\Gamma$ a subset of $\mathcal{B}(X)$. Then
$$ HC(\Gamma)=\bigcap_{n\geq1}\bigcup_{T\in\Gamma}T^{-1}(U_n),$$
where $(U_n)_{n\geq1}$ is a countable basis of the topology of $X$. As a consequence, $HC(\Gamma)$ is a $G_\delta$ type set.
\end{proposition}
\begin{proof}
Suppose that $\Gamma$ is a hypercyclic set. Then, $x\in HC(\Gamma)$ if and only if $\overline{Orb(\Gamma,x)} = X$.
Equivalently, for all $n\geq1$ we have $U_n\cap Orb(\Gamma,x)\neq\emptyset$. That is, for all $n \geq 1$ there exists
$T\in\Gamma$ such that $x \in T^{-1}(U_n)$. This is equivalent to the fact that $x\in \bigcap_{n\geq1}\bigcup_{T\in\Gamma}T^{-1}(U_n)$.
Since $\bigcup_{T\in\Gamma}T^{-1}(U_n)$ is an open subset of X for all $n \geq 1$, it follows that $HC(\Gamma)$ is a $G_\delta$ type.
\end{proof}
\section{Topologically Transitive Sets of Operators}
Topologically transitivity for a single operator was introduced by Birkhoff in \cite{Birkhoff}.
This notion was generalized to sets of operators in \cite{AKH}.
\begin{definition}\cite{AKH}
A set $\Gamma\subset \mathcal{B}(X)$ is said to be topologically transitive if for each pair of nonempty open subsets in $X$ there exists some $T\in \Gamma$ such that
$$T(U)\cap V\neq \emptyset.$$
\end{definition}
\begin{remark}
Let $T\in \mathcal{B}(X)$. If $\Gamma=\{T^n \mbox{ : }n\in\mathbb{N}\}$, then $\Gamma$ is topologically transitive as a set of operators if and only if $T$ is topologically transitive as an operator.
\end{remark}
The topological transitivity of a single operators is preserved under quasi-similarity \cite[Proposition 1.13]{Erdmann Peris}. The following proposition proves that the same result holds of sets of operators.
\begin{proposition}
Let $X$ and $Y$ be topological vector spaces and let $\Gamma\subset\mathcal{B}(X)$ be quasi-similar to $\Gamma_1\subset\mathcal{B}(Y)$.
If $\Gamma$ is topologically transitive in $X$, then $\Gamma_1$ is topologically transitive in $Y$. 
\end{proposition}
\begin{proof}
Since  $\Gamma$ and $\Gamma_1$ are quasi-similar, there exists a continuous map $\phi$ : $X\longrightarrow Y$ with dense range such that for all $T\in\Gamma,$ there exists $S\in\Gamma_1$ satisfying $S\circ\phi=\phi\circ T$.
Let $U$ and $V$ be two nonempty open subsets of $Y$. Since $\phi$ is of dense range, $\phi^{-1}(U)$ and $\phi^{-1}(V)$ are nonempty and open subsets of $X$. If $\Gamma$ is topological transitive in $X$, then there exist $y\in \phi^{-1}(U)$ and $T\in\Gamma$ such that $Ty\in\phi^{-1}(V)$, which implies that $\phi(y)\in U$ and $\phi(Ty)\in V$. Let $S\in\Gamma_1$ such that $S\circ\phi=\phi\circ T$, then $\phi(y)\in U$ and $S\phi(y)\in V$. From this, we deduce that $\Gamma_1$ is topological transitive in $Y$.
\end{proof}
\begin{corollary}
Let $X$ and $Y$ be topological vector spaces and let $\Gamma\subset\mathcal{B}(X)$ be similar to $\Gamma_1\subset\mathcal{B}(Y)$.
Then, $\Gamma$ is topologically transitive in $X$ if and only if $\Gamma_1$ is topologically transitive in $Y$. 
\end{corollary}
In the following result, we give necessary and sufficient conditions for a set of operators to be topologically transitive.
\begin{theorem}\label{tt}
Let $X$ be a complex normed space and $\Gamma\subset \mathcal{B}(X)$. The following assertions are equivalent:
\begin{itemize}
\item[$(i)$] $\Gamma$ is topologically transitive;
\item[$(ii)$] For each $x$, $y\in X,$ there exists sequences $\{x_k\}$ in $X$ and $\{T_k\}$ in $\Gamma$ such that
$$x_k\longrightarrow x\hspace{0.3cm}\mbox{ and }\hspace{0.3cm}T_k( x_k)\longrightarrow y;$$
\item[$(iii)$] For each $x$, $y\in X$ and for $W$ a neighborhood of $0$, there exist $z\in X$ and $T\in\Gamma$  such that
$$T( z)-y\in W \hspace{0.3cm}\mbox{ and }\hspace{0.3cm} x-z\in W. $$
\end{itemize}
\end{theorem}
\begin{proof}$(i)\Rightarrow(ii)$
Let $x$, $y\in X$. For all $k\geq1$, let $U_k=B(x,\frac{1}{k})$ and $V_k=B(y,\frac{1}{k})$. Then $U_k$ and $V_k$ are nonempty open subsets of $X$. Since $\Gamma$ is topologically transitive, there exists $T_k\in\Gamma$ such that $T_k(U_k)\cap V_k\neq\emptyset$. For all $k\geq1$, let $x_k\in U_k$ such that $T_k( x_k)\in V_k$, then
$$\Vert x_k-x \Vert<\frac{1}{k}\hspace{0.3cm}\mbox{ and }\hspace{0.3cm}\Vert T_k(x_k)-y \Vert<\frac{1}{k}$$
which implies that
$$x_k\longrightarrow x\hspace{0.3cm}\mbox{ and }\hspace{0.3cm} T_k(x_k)\longrightarrow y.$$

$(ii)\Rightarrow(iii)$ Let $x$, $y\in X$. There exists sequences $\{x_k\}$ in $X$ and $\{T_k\}$ in $\Gamma$ such that
$$x_k-x\longrightarrow 0\hspace{0.3cm}\mbox{ and }\hspace{0.3cm}T_k(x_k)-y\longrightarrow 0.$$
If $W$ is a neighborhood of $0$, then there exists $N\in\mathbb{N}$ such that $x-x_k\in W$ and $T_k(x_k)-y\in W$, for all $k\geq N$.

$(iii)\Rightarrow(i)$ Let $U$ and $V$ be two nonempty open subsets of $X$. Then there exists $x$, $y\in X$ such that $x\in U$ and $y\in V$. Since for all $k\geq1$,  $W_k=B(0,\frac{1}{k})$ is a neighborhood of $0$,  there exist $z_k\in X$ and $T_k\in\Gamma$ such that
$$\Vert T_k( z_k)-y\Vert<\frac{1}{k}\hspace{0.3cm}\mbox{ and }\hspace{0.3cm}\Vert x-z_k\Vert<\frac{1}{k}.$$
This implies that
$$z_k\longrightarrow x\hspace{0.3cm}\mbox{ and }\hspace{0.3cm}T_k(z_k)\longrightarrow y.$$
 Since $U$ and $V$ are nonempty open subsets of $X$, $x\in U$ and $y\in V$, there exists $N\in\mathbb{N}$ such that $z_k\in U$ and $T_k(z_k)\in V$, for all $k\geq N.$
\end{proof}

It is known from Birkhoff's transitivity theorem \cite{Birkhoff} that an operator $T \in \mathcal{B}(X)$ is hypercyclic if and only if it's topologically transitive.
Let $\Gamma$ be a subset of $\mathcal{B}(X).$ In what follows, we prove that
 $\Gamma$ is topologically transitive implies that $\Gamma$ is hypercyclic.
\begin{theorem}\label{th3}
Let $X$ be a second countable Baire complex topological vector space and $\Gamma$ a subset of $\mathcal{B}(X)$. Then, the following assertions are equivalent:
\begin{itemize}
\item[$(i)$]$HC(\Gamma)$ is dense in $X$;
\item[$(ii)$]$\Gamma$ is topologically transitive.
\end{itemize}
As a consequence, a topologically transitive set is hypercyclic.
\end{theorem}
\begin{proof}
Since $X$ is a second countable topological vector space, we can consider $(U_m)_{m\geq1}$ a countable basis of the topology of $X$.\\
$(i)\Rightarrow (ii) : $ Assume that $HC(\Gamma)$ is dense in $X$. By Proposition \ref{5}, we have
$$HC(\Gamma)=\bigcap_{n\geq1}\bigcup_{T\in\Gamma}T^{-1}(U_n).$$
Hence, for all $n\geq 1$, the set $A_n=\bigcup_{T\in\Gamma}T^{-1}(U_n)$ is dense in $X$. As a consequence, for all $n$, $m\geq 1$ we have $A_n\cap U_m\neq\emptyset$. Thus, for all $n$, $m\geq 1$ there exists $T\in \Gamma$ such that $T( U_m)\cap U_n\neq\emptyset$. Since $(U_m)_{m\geq1}$ is a countable basis of the topology of $X$, it follows that $\Gamma$ is topologically transitive.\\
$(ii)\Rightarrow (i) : $ Assume that $\Gamma$ is a topologically transitive set. Let $n$, $m\geq 1$, then there exists $T\in \Gamma$ such that $T(U_m)\cap U_n\neq \emptyset$, that is  $T^{-1}(U_n)\cap U_m\neq \emptyset$. Hence, for all $n\geq1$, the set $\bigcup_{T\in\Gamma}T^{-1}(U_n)$ is dense in $X$. Since $X$ is a Baire space, it follows that
$HC(\Gamma)=\bigcap_{n\geq1}\bigcup_{T\in\Gamma}T^{-1}(U_n)$
is dense in $X$.
\end{proof}
The converse of Theorem \ref{th3} holds with additional assumptions.
\begin{theorem}\label{so}
Assume that $X$ is without isolated point and let $\Gamma\subset\mathcal{B}(X)$ such that for all $T$, $S\in\Gamma$ with $T\neq S$, there exists $A\in\Gamma$ such that $T=AS$. Then $\Gamma$ is hypercyclic implies that $\Gamma$ is topologically transitive.
\end{theorem}
\begin{proof}
Since $X$ is without isolated point,  we can suppose that $I\in \Gamma$.
Since $\Gamma$ is a hypercyclic set, there exists $x\in X$ such that $Orb(\Gamma,x)$ is dense in $X$.
Let $U$ and $V$ be two nonempty open sets of $X$, then there exist $T$, $S\in\Gamma$ such that
\begin{equation}\label{e1}
 Tx\in U \hspace{0.2cm}\mbox{ and }\hspace{0.2cm} Sx\in V.
\end{equation}
If $T=S$, then $U\cap V\neq \emptyset$ which means that $I(U)\cap V\neq\emptyset$. Since $I\in\Gamma$, the result holds.\\
If $T\neq S$, then there exists $A\in\Gamma$ such that $T=AS$. By (\ref{e1}), we have
$$ A(Sx)\in U \hspace{0.2cm}\mbox{ and }\hspace{0.2cm} A(Sx)\in  A(V),$$
this means that $U\cap A(V)\neq \emptyset$.
Hence, $\Gamma$ is a topologically transitive set.
\end{proof}
\begin{remark}
Assume that $X$ is not necessary without isolated point and let $\Gamma\subset\mathcal{B}(X)$ be a hypercyclic set. If $\Gamma$ satisfies the condition of Theorem \ref{so}, then $\Gamma\cup\{I\}$ is  a topologically transitive set.
\end{remark}
The following example shows that the condition of Theorem \ref{so} is sufficient and not necessary.
\begin{example}\label{2}
Let $f$ be a nonzero linear form in $X$. Then there exists $e\in X\setminus\{0\}$ such that $f(e)\neq 0$. For all $x\in X$ let $T_x$ be an operator defined by
$$\begin{array}{ccccc}
T_x & : & X & \longrightarrow & X \\
 & & y & \longmapsto & \frac{f(x) }{ f(e) }x. \\
\end{array}$$
Let $\Gamma=\{T_x\mbox{ : }x\in X\}$. For all $x\in X$, we have $T_x(e)=x$, that is $x\in Orb(\Gamma,e)$. Hence
$$X=\overline{Orb(\Gamma,e)}.$$
Consequently; $\Gamma$ is a hypercyclic set and $e\in HC(\Gamma)$.
Moreover, for all $y\in X\setminus\{0\}$ we have
$$  \overline{Orb(\Gamma,y)}=\overline{\{ T_x(y) \mbox{ : }x\in X\}}=\overline{\left\lbrace \frac{f(y)}{f(e)}x\mbox{ : }x\in X\right\rbrace }=X. $$
Hence $y\in HC(\Gamma)$. Thus, $HC(\Gamma)$ is dense in $X$. By Theorem \ref{th3}, we deduce that $\Gamma$ is a topologically transitive set.
Now Let $e_1$, $e_2$ $\in X\setminus\{0\}$ such that $f(e_2)=0$. If there exists $x\in X$ such that $T_{e_1}=T_x T_{e_2}$, then $T_{e_{1}}(e)=T_x (T_{e_2}(e))$. We have $T_{e_{1}}(e)=e_1$ and
$$T_x (T_{e_2}(e))=T_x(e_2)= \frac{f( e_2) }{f(e)}x=0.$$
This is a contradiction.
\end{example}
In the next definition, we introduce the hypercyclic criterion for sets of operators.
\begin{definition}\label{cc}
Let $\Gamma\subset\mathcal{B}(X)$. We say that $\Gamma$ satisfies the criterion of hypercyclicity if there exist two dense subsets $X_0$ and $Y_0$ in $X$, a sequence $\{k\}$ of positives integers, a sequence of operators $\{T_k\}$ of $\Gamma$, and a sequence of maps $S_k$ : $Y_0\longrightarrow X$ such that:
\begin{itemize}
\item[$(i)$] $ T_kx\longrightarrow 0$ for all $x\in X_0$;
\item[$(ii)$] $ S_kx\longrightarrow 0$ for all $y\in Y_0$;
\item[$(iii)$] $T_kS_ky\longrightarrow y$ for all $y\in Y_0$.
\end{itemize}
\end{definition}
\begin{remark}
Let $T\in \mathcal{B}(X)$. If $\Gamma=\{T^n \mbox{ : }n\in\mathbb{N}\}$, then $\Gamma$ satisfies the hypercyclicity criterion as a set of operators if and only if $T$ satisfies the hypercyclicity criterion as an operator.
\end{remark}
\begin{theorem}\label{11}
Let $X$ be a second countable Baire complex topological vector space and $\Gamma$ a subset of $\mathcal{B}(X)$. If $\Gamma$ satisfies the criterion of hyercyclicity, then $\Gamma$ is topologically transitive. As a consequence, $\Gamma$ is hypercyclic.
\end{theorem}
\begin{proof}
Let $U$ and $V$ be two nonempty open sets of $X$. Since $X_0$ and $Y_0$ are dense in $X$, there exist $x_0$ and $y_0$ in $X$ such that
$$ x_0\in X_0\cap U\hspace{1cm}\mbox{ and }\hspace{1cm}y_0\in Y_0\cap V. $$
For all $k\geq1$, let $z_k=x_0+S_ky $. By Definition \ref{cc} we have
$S_k y\longrightarrow 0$ which implies that $z_k\longrightarrow x_0$. Since $x_0\in U$
 and $U$ is open, there exists $N_1\in\mathbb{N}$ such that $z_k\in U$, for all $k\geq N_1$. On the other hand, $ T_k z_k=T_k x_0+T_k (S_k y_0)\longrightarrow y_0$. Since $y_0\in V$ and $V$ is open, there exists $N_2\in\mathbb{N}$ such that $T_k z_k\in V$, for all $k\geq N_2$. Let $N=\max\{N_1,N_2\}$, then  $z_k\in U$ and $ T_k z_k\in V$ for all $k\geq N$. Thus,
$$ T_k (U)\cap V\neq \emptyset,$$
for all $k\geq N$. Hence, $\Gamma$ is topologically transitive set.
\end{proof}
\begin{remark}
If $X$ is a complex normed space and $\Gamma$ is a subset of $\mathcal{B}(X)$ which satisfies the criterion of hypercyclicity, then by Theorem \ref{11} $\Gamma$ is a topologically transitive set. Thus, $\Gamma$ satisfies the conditions $(ii)$ and $(iii)$ of Theorem \ref{tt}.
\end{remark}
\section{Application}

In this section, we study the particular case where $\Gamma$ stands for a $C$-regularized semigroup.
Reccal \cite{CKM}, that an entire $C$-regularized group is an operator family $(S(z))_{z\in\mathbb{C}}$ on $\mathcal{B}(X)$ that satisfies:
\begin{itemize}
\item[$(1)$] $S(0)=C;$
\item[$(2)$] $S(z+w)C = S(z)S(w)$ for every $z,$ $w\in\mathbb{C}$,
\item[$(3)$] The mapping $z \mapsto S(z)x$, with $z\in\mathbb{C}$, is entire for every $x \in X$.
\end{itemize}
\begin{example}\label{ex1}
Let $X=\mathbb{C}$. For all $z\in\mathbb{C}$, let $S(z)x=\exp(z)x$, for all $x\in \mathbb{C}$. $(S(z))_{z\in\mathbb{C}}$ is a $C$-regularized group and we have
$\overline{Orb((S(z))_{z\in\mathbb{C}},1)}=\overline{\{\exp(z) \mbox{ : }z\in\mathbb{C}\}}=\mathbb{C},$ which implies that $(S(z))_{z\in\mathbb{C}}$ is hypercyclic.
\begin{remark}
Example \ref{ex1} shows that there exists a hypercyclic $C$-regularized group in finite dimensional. Moreover, the fact that the $C$-regularized group is hypercyclic does not implies that each $S(z)$ is a hypercyclic operator, for all $z\in \mathbb{C}$. This, is since finite dimensional space support no hypercyclic operator \cite{Rolewicz}.
\end{remark}
\end{example}
\begin{lemma}\label{lem}
Let $(S(z))_{z\in\mathbb{C}}$ be a hypercyclic $C$-regularized group. If $C$ has dense range, then $Cx\in HC((S(z))_{z\in\mathbb{C}})$, for all $x\in HC((S(z))_{z\in\mathbb{C}}).$
\end{lemma}
\begin{proof}
This, since $C\in \{(S(z))_{z\in\mathbb{C}}\}^{'}$.
\end{proof}

By Theorem \ref{th3}, every topological transitive $C$-regularized group is hypercyclic. In the flowing we prove that the converse holds.
\begin{theorem}
Let $(S(z)_{z\in\mathbb{C}})$ be a $C$-regularized group such that $C$ has dense range. If $(S(z)_{z\in\mathbb{C}})$  is hypercyclic, then $(S(z)_{z\in\mathbb{C}})$ is topologically transitive.
\end{theorem}
\begin{proof}
Let $x\in HC((S(z))_{z\in\mathbb{C}})$, then by Lemma \ref{lem}, $Cx\in HC((S(z))_{z\in\mathbb{C}}).$
Let $U$ and $V$ be two nonempty open subsets of $X$. There exist $z_1$, $z_2\in\mathbb{C}$ such that
\begin{equation}\label{e11}
S(z_1)x\in C^{-1}(U) \hspace{0.6cm}\mbox{ and }\hspace{0.6cm} S(z_2)x\in V.
\end{equation}
Let $z_3=z_1-z_2$. By $\ref{e11}$, we have
$$ S(z_3)(S(z_2)x)\in U\hspace{0.3cm}\mbox{ and }\hspace{0.3cm} S(z_3)(S(z_2)x)\in  S(z_3)(V),$$
which implies that $U\cap S(z_3)(V)\neq \emptyset$.
Hence, $(S(z))_{z\in\mathbb{C}}$ is a topologically transitive $C$-regularized group.
\end{proof}

If $X$ is a Banach infinite dimensional space, then we have the next theorem.
\begin{theorem}
Let $(S(z))_{z\in\mathbb{C}}$ be a $C$-regularized group on a Banach infinite dimensional space $X$. If $x \in X$ is a hypercyclic vector of $(S(z))_{z\in\mathbb{C}}$, then the set $\{S(z) x \mbox{ : }\vert z \vert\geq \vert w \vert\}$ is dense in $X$, for all $w\in\mathbb{C}$.
\end{theorem}
\begin{proof}
Suppose that there exists $w_0\in\mathbb{C}$ such that $A=\{S(z) x \mbox{ : }\vert z \vert\geq \vert w_0\vert\}$ is not dense in X. Hence there exists a bounded open set $U$ such that $U \cap \overline{A}= \emptyset$.
Therefore we have
$$ U\subset \overline{\{S(z) x \mbox{ : }0\leq\vert z \vert\leq \vert w_0\vert\}} $$
by using the relation
$$ X=\overline{\{S(z) x \mbox{ : }z\in\mathbb{C}\}}=\overline{\{S(z) x \mbox{ : }\vert z \vert\geq \vert w_0\vert\}}\cup \overline{\{S(z) x \mbox{ : }0\leq\vert z \vert\leq \vert w_0\vert\}}, $$
which means that $\overline{U}$ is compact. Hence $X$ is finite dimensional, which contradicts that $X$ is infinite dimensional.
\end{proof}


\begin{thebibliography}{00}
\bibitem {AKH} Ansari M, Khani-robati B, Hedayatian K.
On the density and transitivity of sets of operators. Turk J Math 2018; 42: 181-189.
\bibitem{Bayart Matheron} Bayart F, Matheron E. Dynamics of Linear Operators. New York, NY, USA: Cambridge University Press, 2009.
\bibitem{Bes} B$\grave{e}$s JP.
Three problem's on hypercyclic operators. Ph.D. thesis, Bowling Green State University, Bowling Green, Ohio, 1998.
\bibitem{Birkhoff} Birkhoff G. Surface transformations and their dynamical applications. Acta Math 1922; 43: 1-119.
\bibitem{BMP} Bonet J, Martinez-Gim$\acute{e}$nez F, Peris A.
Linear chaos on Frechet spaces. Internat J Bifur Chaos Appl Sci Engrg 2003; 13: 1649-1655.
\bibitem{Bourdon Feldman}
Bourdon DS, Feldman NS.
Somewhere dense orbits are everywhere dense. Indiana Unive Math J 2005; 52: 811-819.
\bibitem{CKM} Conejero JA, Kostic M, Miana PJ, Murillo-Arcila M. Distributionally chaotic families
of operators on Frechet spaces. Commun Pure Appl Anal 2016; 15: 1915-193.
\bibitem{DR} De la Rosa M, Read C.
A hypercyclic operator whose direct sum $T\oplus T$ is not hypercyclic. J Operator Theory 2009; 61: 369-380.
\bibitem{Feldman}  Feldman NS.
The dynamics of cohyponormal operators. Trends in Banach spaces and operator theory Memphis TN Contemp Math 2001; 321: 71-85.
\bibitem{Gethner Shapiro} Gethner RM, Shapiro JH.
Universal vectors for operators on spaces of holomorphic functions. Proc Amer Math Soc 1987; 100: 281-288.
\bibitem{Godefroy Shapiro}Godefroy G, Shapiro J. Operators with dense invariant cyclic vector manifolds. J Funct Anal 1991; 98: 229-269.
\bibitem{GED}  Grosse-Erdmann KG.
Dynamics of linear operators. Topics in complex analysis and operator theory Univ  Malaga 2007; 41-84.
\bibitem{Erdmann Peris} Grosse-Erdmann KG, Peris A.
Linear chaos. Universitext, Springer, London, 2011.
\bibitem{GEU} Grosse-Erdmann KG.
Universal families and hypercyclic operators. Bull Amer Math Soc 1999; 36: 345-381.
\bibitem{HL} Herrero DA.
Limits of hypercyclic and supercyclic operators. Fund Anal 1991; 99: 179-190.
\bibitem{Kitai} Kitai C.
Invariant closed sets for linear operators. Ph.D. thesis, University of Toronto, Toronto, 1982.
\bibitem{MS} Montes-Rodriguez A, Salas HN.
Supercyclic subspaces: spectral theory and weighted shifts. Adv Math 2001; 163: 74-134.
\bibitem{Muller} M$\ddot{u}$ller V.
Spectral theory of linear operators and spectral systems in Banach algebras.  Oper Theory Adv Appl Birkh$\ddot{a}$user Verlag Basel 2003; 139.
\bibitem{Rolewicz} Rolewics S. On orbits of elements. Studia Math 1969; 32: 17-22.
\bibitem{Salas} Salas HN,
 A hypercyclic operator whose adjoint is also hypercyclic. Proc Amer Math Soc 1991; 112: 765-770.
\bibitem{Willered}
 Willard S. General Topology. Addison-Wesley, Pub, Co, 1970.
\end{thebibliography}
\end{document}